\documentclass[10 pt]{amsart}

\usepackage{amsmath} 
\usepackage{amsfonts}
\usepackage{amssymb}
\usepackage{amstext}
\usepackage{amsbsy}
\usepackage{amsopn}
\usepackage{amsthm}
\usepackage{amsxtra}
\usepackage{graphicx}
\usepackage{caption}
\usepackage{subcaption}
\usepackage{color}
\usepackage{hyperref}
\usepackage{enumerate}
\usepackage{amscd}
\usepackage{tikz}

\newtheorem{theorem}{Theorem}[section]

\newtheorem{proposition}[theorem]{Proposition}

\newtheorem*{conjecture*}{Conjecture}
\newtheorem*{corollary*}{Corollary}
\newtheorem*{claim*}{Claim}
\newtheorem*{theorem*}{Theorem}

\theoremstyle{remark}

\theoremstyle{definition}

\newcommand{\one}{\mathbf{1}}

\newcommand{\A}{\mathcal{A}}

\newcommand{\R}{\mathbb{R}}
\newcommand{\Z}{\mathbb{Z}}

\newcommand{\N}{\mathbb{N}}

\newcommand{\CA}{\mathcal{A}}
\newcommand{\CI}{\mathcal{I}}
\newcommand{\CL}{\mathcal{L}}

\newcounter{note}

\title{Counting generic measures for a subshift of linear growth}
\author{Van Cyr}
\address{Bucknell University, Lewisburg, PA 17837 USA}
\email{van.cyr@bucknell.edu}
\author{Bryna Kra}
\address{Northwestern University, Evanston, IL 60208 USA}
\email{kra@math.northwestern.edu}

\subjclass[2010]{}


\thanks{The second author was partially supported by NSF grant.}

\begin{document}

\begin{abstract}
In 1984 Boshernitzan proved an upper bound on the number of ergodic measures for a minimal subshift of linear block growth and asked if it could be lowered without further assumptions on the shift.  We answer this question, showing that Boshernitzan's bound is 
sharp.  We further prove 
that the same bound holds for the, a priori, larger set of nonatomic generic measures, and that this bound 
remains valid even if one drops the assumption of minimality.  Applying these results to interval exchange transformations, 
we give an upper bound on the number of nonatomic generic measures of a minimal IET, answering a question recently 
posed by Chaika and Masur. 
\end{abstract}

\maketitle

\section{Introduction}

Let $(X, \sigma)$ be a subshift, meaning that $X\subset\CA^\Z$, where $\CA$ is a finite alphabet, and $X$ is a closed set that is invariant under the left shift $\sigma\colon \CA^\Z\to\CA^\Z$.  
A classic problem is to find conditions
that imply $(X,\sigma)$ is uniquely ergodic or, more generally, has a finite number of ergodic measures.    
In the 1980's, 
Boshernitzan~\cite{Bos} showed that the complexity of the subshift can be used to obtain such a result.  More precisely, 
if $P_X(n)$ is the number of words of length $n$ which occur in any $x\in X$, he showed that if $(X,\sigma)$ is minimal and 
$\limsup_{n\to\infty}P_X(n)/n < 3$, then it is uniquely ergodic (see also related results in~\cite{Bos2}).  More generally, 
Boshernitzan showed that if 
\begin{equation}
\label{eq:complexity}
\liminf_{n\to\infty}\frac{P_X(n)}{n}<k,
\end{equation} then there are at most $k-1$ ergodic measures.  
Some motivation for studying this problem is generalizing the well-known bound on the number of ergodic measures for an interval exchange transformation (IET), that had been previously proven, independently, by Katok and Veech.  Boshernitzan's Theorem applies to a much broader class of dynamical systems than the interval exchange transformations, but the bound 
he obtains is weaker than that of Katok and Veech in the case of an IET.  
Boshernitzan asked in~\cite{Bos}, and then again in~\cite{Bos1}, whether his bound could be lowered in this more general setting.  One of our main results answers Boshernitzan's question: for the class of minimal subshifts whose complexity function satisfies~\eqref{eq:complexity}, Boshernitzan's bound is a sharp bound for the number of nonatomic ergodic measures.  Our technique also shows that the bound is more general than originally stated: the same bound remains valid (and sharp) even without the assumption of minimality and even if one seeks to bound the (a priori, larger) set of nonatomic generic measures.

The particular case of minimal interval exchange 
transformations has been well studied (for example Katok~\cite{K}, Keane~\cite{Keane}, and Veech~\cite{V}).  A minimal $k$-interval exchange transformation ($k$-IET) has a natural symbolic cover, its natural coding, and this subshift satisfies the hypothesis of Boshernitzan's Theorem.  As an application, 
this shows that a minimal $k$-IET (see Section~\ref{sec:IET} for the definition) has at most $k-1$ ergodic measures.  
The optimal bound of $\lfloor k/2\rfloor$ was proven, independently, 
by Katok~\cite{K} and Veech~\cite{V}.  In a recent paper, Chaika and Masur~\cite{CM} studied the broader class of generic measures for an IET and asked whether there are 
bounds on the number of such measures.  An interesting facet of this problem is that 
although several quite different proofs of the bound given by Katok and Veech for the number of ergodic measures exist in the literature, they all use ergodicity in an essential way. 

If $X$ is a compact metric space, $\mathcal B$ the Borel $\sigma$-algebra, $\mu$  a Borel probability measure on $\mathcal B$,
and $T\colon X\to X$ is a measurable map preserving the measure $\mu$, a point $x\in X$ is a
{\em generic point for the measure $\mu$}  if for every continuous function $f\colon X\to\R$,
$$
\lim_{N\to\infty}\frac{1}{N}\sum_{n=0}^{N-1}f(T^nx)= \int f\, d\mu.  
$$
The measure $\mu$ is {\em generic} if it has a generic point.
Thus, by the Pointwise Ergodic Theorem, if the measure $\mu$ is ergodic almost every point is generic.  However, a generic 
measure need not be ergodic.  Chaika and Masur~\cite{CM} constructed a $6$-interval exchange transformation 
that has a generic, but not ergodic, measure.  They asked if there is a bound on the number of generic measures for a $k$-IET.  We show: 
\begin{theorem}
\label{th:upper}
If $(X, \sigma)$ is a subshift and there exists $k\in\N$ such that 
$$
\liminf_{n\to\infty}\frac{P_X(n)}{n} < k,
$$
then $(X,\sigma)$ has at most $k-1$ distinct, nonatomic, generic measures.  
\end{theorem}
In particular, this applies to interval exchange transformations by passing to the natural cover.  
Theorem~\ref{th:upper} generalizes Boshernitzan's Theorem~\cite{Bos} in two ways:  
there is no assumption of minimality and our bound holds for the more general class of generic measures.  We also give an analogous bound for $\limsup$ (note the technical assumption is vacuous for minimal subshifts that are not uniquely ergodic). 
\begin{theorem}\label{th:upper-sup} 
Suppose $(X, \sigma)$ is a subshift and there exists $k\in\N$ such that 
$$
\limsup_{n\to\infty}\frac{P_X(n)}{n} < k.  
$$
If $(X,\sigma)$ has a generic measure $\mu$ and a generic point $x_{\mu}$ for which the orbit closure 
$$ 
\overline{\{\sigma^kx_{\mu}\colon k\in\Z\}} 
$$ 
is not uniquely ergodic, then $(X,\sigma)$ has at most $k-2$ distinct, nonatomic, generic measures. 
\end{theorem} 
Recently Damron and Fickenscher~\cite{DF} proved a related result, showing that any minimal 
shift $(X, \sigma)$ whose complexity function satisfies $P_X(n) = kn+c$ for some constant $c$, 
$k\geq 4$, and all $n$ sufficiently large has at most $k-2$ ergodic measures.  

Moreover, we show that these theorems are sharp, even if $X$ is assumed to be minimal and the measures are required to be ergodic. 
\begin{theorem}\label{th:sharp} 
Suppose $d>1$ is an integer.  There exists a minimal subshift $(X,\sigma)$ which has exactly $d$ ergodic measures, zero nonergodic generic measures, and which satisfies 
\begin{eqnarray*} 
\liminf_{n\to\infty}\frac{P_X(n)}{n}&=&d; \\ 
\limsup_{n\to\infty}\frac{P_X(n)}{n}&=&d+1. 
\end{eqnarray*} 
\end{theorem} 
 
We include several other examples in Section~\ref{section:sharp}, showing other senses in which Theorems~\ref{th:upper} and~\ref{th:upper-sup} can be said to be sharp.  

As an application of Theorem~\ref{th:upper}, we answer Chaika and Masur's question:  
\begin{theorem}\label{thm:1.1}
For $k> 2$, a minimal $k$-interval exchange transformation has at most $k-2$ generic measures.
\end{theorem} 
For $k=2$, a minimal $2$-interval exchange is an ergodic rotation, which is uniquely ergodic.  
For $k=3$ and $4$, Theorem~\ref{thm:1.1} is sharp upper bound, but we do not know if it is sharp for $k\geq 5$.  In particular, we do not know
if we can improve the symbolic result of Theorem~\ref{th:upper} for systems that arise
as the natural coding of an interval exchange transformation.  We also do not know if there can be a second  generic measure in the example of Chaika and Masur, nor if a $6$-interval exchange 
with three ergodic measures can also have a generic (and obviously nonergodic) measure.  

\section{Background and notation}

If $\A$ is a finite alphabet, a {\em word} $w$ in the alphabet is a concatenation of letters in $\A$ and
the length $|w|$ of the word is the number (finite or infinite) of letters.  A {\em word $w = w_1\ldots w_\ell$ occurs  in a word $u = u_1\ldots u_k$} if
there is some $m\in\{1, \ldots, k-\ell\}$ such that $w_1=u_m, \ldots, w_\ell = u_{m+\ell}$, and we refer to $w$ as a {\em subword} of $u$.
The analogous definitions hold for a finite word $w$ occurring as a subword of an infinite word $u$.

A {\em language} $\mathcal L$ is a set of (finite) words such that if $w\in\mathcal L$, then any subword is also contained
in $\mathcal L$.  The language {\em determined by a word} (finite or infinite) is the collection of all finite subwords of the
word.  We let $\mathcal L_n$ denote all the words in the language $\mathcal L$ of length $n$.  If $w\in\mathcal L$, we write $[w]$ for the {\em cylinder set determined by $w$},
meaning that
$$[w] = \{u\in \mathcal L\colon \text{ the first } |w| \text{ symbols of } u \text{ agree with } w\}.
$$

We assume that the alphabet $\A$ is endowed with the discrete topology and if $x\in\CA^\Z$, we use 
$x(n)$ to denote the value of $x$ at $n\in\Z$.  The space $\A^\Z$ is a compact metric space when endowed with the product topology (and a compatible metric). 

A {\em subshift $(X,\sigma)$} is a closed subset $X\subset \A^{\Z}$ that is invariant under the left shift $\sigma\colon \A^{\Z}\to \A^{\Z}$ defined by $(\sigma x)(n) = x(n+1)$.
If $\mathcal L$ is the language of the system $X$, meaning the set of all finite subwords that arise
for any $x\in X$, we write $\mathcal L = \mathcal L(X)$ and we write $\mathcal L _n = \mathcal L_n(X)$ for the words of length $n$.
We define the {\em complexity function} $P_X\colon X\to \N$ by
$$
P_X(n) = |\mathcal L_n(X)|.
$$

For a word $w\in \mathcal L(X)$, we write $\one_{[w]}$ for the indicator function of the word $w$.
We say that $x = \bigl(x(n)\bigr)_{n\in\Z}\in X$ is {\em periodic} if there exists $m\neq 0$ such that $x(m+n) = x(n)$ for all $n\in\Z$ 
and otherwise it is {\em aperiodic}.  The point $x$ is {\em eventually periodic} if there exists $m\neq 0$ and $N\in\N$ such that  $x(m+n) = x(n)$ for all $n\geq N$.

For a a system $(X, \sigma)$, the {\em orbit} of $x\in X$ is defined to be $\{\sigma^n x\colon n\in\Z\}$ and the system
 is {\em minimal} if the orbit closure $ \overline{\{\sigma^nx\colon n\in\Z\}}=X$ for any $x\in X$.

We make use of the following theorem (though stated differently) of Epifanio, Koskas, and Mignosi~\cite{EKM}:
\begin{theorem}[{\cite[Theorem 2.2]{EKM}}]
\label{theorem:qz-revised}
Assume $x\in\mathcal{A}^{\N}$ is not eventually periodic and  
fix $M,N_0\in\N$.  Suppose 
that for some $N\geq N_0$, there exist $M\leq m_1<m_2\leq N$ such that $w_x(N,m_1)=w_x(N,m_2)$, 
where 
$$
w_x(N,m):=(x(m), x(m+1), x(m+2), \dots, x(m+N-1)).  
$$
Then there exists $K\geq m_2$ such that 
\begin{enumerate}
	\item (Distinct Words Condition):  for all $K\leq k_1<k_2\leq K+N-N_0$ we have $w_x(N,k_1)\neq w_x(N,k_2)$;
	\item (Prefix First Occurrence Condition): for all $K\leq k<K+N-N_0$ there exists $M\leq l_k\leq N$ such that $w_x(N_0,k)=w_x(N_0,l_k)$.  
\end{enumerate}
\end{theorem}

For completeness, we include the proof, but it is merely a translation of the proof in~\cite{EKM} using 
our hypotheses and emphasizing the stronger conclusion. 
\begin{proof}
Suppose $w_x(N,m_1)=w_x(N,m_2)$.  Then the word $w_x(N+m_2-m_1,m_1)$ is periodic of period $m_2-m_1$.  Since $x$ is not eventually periodic, there exists $N^{\prime}\geq N+m_2-m_1$ such that $w_x(N^{\prime},m_1)$ is periodic of period $m_2-m_1$, but $w_x(N^{\prime}+1,m_1)$ is not.  Let $1\leq p\leq m_2-m_1$ be the minimal period of $w_x(N^{\prime},m_1)$ and define $m_3:=m_1+N^{\prime}-N-p$.  By minimality of $p$ and the fact that $N\geq p$, if $m_3\leq i<j<m_3+p-1$ then $w_x(N,i)\neq w_x(N,j)$.  

For contradiction, suppose there exist $m_3\leq i<j\leq m_3+N$ such that $w_x(N,i)=w_x(N,j)$.  
Since $i,j$ cannot both be smaller than $m_3+p$, it follows that $j\geq m_3+p$.
The word $w_x(N+(j-i),i)$ is periodic of period $j-i$ and its prefix of length $p+j-i$ is periodic of period $p$.  By the Fine-Wilf Theorem~\cite{FW}, 
it follows that this prefix is periodic of period $\gcd(j-i,p)$.  Since this prefix has length at least $p$, it follows that $w_x(N+(j-i),i)$ is periodic of period $\gcd(j-i,p)$ and, in particular, is periodic of period $p$.  
But $w_x(N^{\prime}+1,m_1)$ is not periodic of period $p$, 
by the definition of $N^{\prime}$.  
This contradiction implies that $w_x(N,i)\neq w_x(N,j)$ for any $M\leq i<j\leq m_3+n-N_0$.

Since $w_x(N^{\prime},m_1)$ is periodic of period $p\leq n$ and the length $N_0$ prefix of $w_i(N,i)$ is a subword of $w_x(N^{\prime},m_1)$, the second statement follows.  
\end{proof}

\section{Main results} 
Theorems~\ref{th:upper} and~\ref{th:upper-sup} follow from the following estimate: 
\begin{theorem}\label{theorem:main}
Let $(X,\sigma)$ be a subshift which has at least $d\geq 1$ distinct, nonatomic, generic measures.  Then 
$$ 
\liminf_{n\to\infty}\frac{P_X(n)}{n}\geq d. 
$$ 
If, in addition, $(X,\sigma)$ has a generic measure $\mu$ and a generic point $x_{\mu}$ whose orbit closure 
$$
\overline{\{\sigma^kx_{\mu}\colon k\in\N\}} 
$$ 
is not uniquely ergodic, then 
$$ 
\limsup_{n\to\infty}\frac{P_X(n)}{n}\geq d+1. 
$$ 
\end{theorem}

\begin{proof} 
We show that for arbitrarily small $\delta>0$, we have 
\begin{equation}\label{eq:delta1} 
\liminf_{n\to\infty}\frac{P_X(n)}{n}>d-2d\delta 
\end{equation}
and, under the additional hypothesis of a generic measure and associated generic 
point whose orbit closure is not uniquely ergodic, 
\begin{equation}\label{eq:delta2} 
\limsup_{n\to\infty}\frac{P_X(n)}{n}>d+1-2d\delta. 
\end{equation} 
The theorem follows immediately from these estimates.  

Fix $\delta>0$, and for convenience assume that $1/\delta\in\N$. 
Suppose $\mu_1,\dots,\mu_d$ are distinct, nonatomic, generic measures for $(X,\sigma)$ and choose 
$x_1,\dots,x_d\in X$ such that for each $1\leq i\leq d$, $x_i$ is generic for $\mu_i$.  Since $\mu_i$ is nonatomic, $x_i$ is not eventually periodic.  By definition of $x_i$, for all $w\in\mathcal{L}(X)$ we have 
\begin{equation}\label{eq:limit} 
\lim_{N\to\infty}\frac{1}{N}\sum_{k=0}^{N-1}1_{[w]}(T^kx_i)=\mu_i([w]). 
\end{equation} 
For $1\leq j_1<j_2\leq d$, choose words $w_{(j_1,j_2)}\in\mathcal{L}(X)$ such that $\mu_{j_1}([w_{(j_1,j_2)}])\neq\mu_{j_2}([w_{(j_1,j_2)}])$.  Set 
\begin{equation}\label{eq:epsilon} 
\varepsilon:=\min\{|\mu_{j_1}([w_{(j_1,j_2)}])-\mu_{j_2}([w_{(j_1,j_2)}])|\colon1\leq j_1<j_2\leq d\} 
\end{equation} 
and set 
\begin{equation}\label{eq:B} 
B:=\frac{\delta}{16-4\delta}. 
\end{equation} 
By~\eqref{eq:limit}, for each $1\leq i\leq d$ there exists $N_i\in\N$ such that for all $N\geq N_i$ and all $1\leq j_1<j_2\leq d$, we have 
\begin{equation}\label{eq:distance} 
\left|\frac{1}{N}\sum_{k=0}^{N-1}1_{[w_{(j_1,j_2)}]}(T^kx_i)-\mu_i([w_{(j_1,j_2)}])\right|<B\cdot\varepsilon. 
\end{equation} 
Set 
\begin{equation}\label{eq:M} 
M:=\max_{1\leq i\leq d}N_i. 
\end{equation} 

For $1\leq i\leq d$ and for $N,m\in\N$, define $u_i(N,m)\in\mathcal{L}_N(X)$ by 
$$
u_i(N,m):=(x_i(m),x_i(m+1),x_i(m+2),\dots,x_i(m+N-1)) 
$$
to be the word of length $N$ that occurs in $x_i$ starting at location $m$. 

If $u,w\in\mathcal{L}(X)$ and $|u|\geq|w|$, define the frequency with which $w$ occurs as a subword in $u$ to be
\begin{equation}\label{eq:frequency} 
F(u,w):=\frac{1}{|u|-|w|+1}\sum_{k=0}^{|u|-|w|}1_{[w]}(T^kx), 
\end{equation} 
where $x\in[u]$.  Note that this frequency does not depend on the choice of $x\in[u]$, as it only depends on the first $|u|$ coordinates of $x$.  Suppose 
$$ 
N\geq\frac{1}{\delta}\cdot(M+\max\{|w_{(j_1,j_2)}|\colon1\leq j_1<j_2\leq d\}) 
$$ 
is fixed and define $L_N:=\lfloor(2-\delta)N\rfloor$ and $\ell_N:=\lfloor\delta N\rfloor$.  
By definition, $\ell_N-|w_{(j_1,j_2)}|\geq M$ for all $w_{(j_1,j_2)}$.  If $1\leq i\leq d$, $1\leq j_1<j_2\leq d$, and $M\leq L\leq L_N$, then the frequency with which the word $w_{(j_1,j_2)}$ occurs in the subword of $x_i$ with length $\ell_N$ and starting from location $L$ is  given by
	\begin{eqnarray*} 
	F\left(u_i(\ell_N,L),w_{(j_1,j_2)}\right) \\ 
	&&\hspace{-1.5 in}=\frac{1}{\ell_N-|w_{(j_1,j_2)}|+1}\sum_{k=0}^{\ell_N-|w_{(j_1,j_2)}|}1_{[w_{(j_1,j_2)}]}(T^k(T^Lx_i)) \\ 
	&&\hspace{-1.5 in}=\frac{1}{\ell_N-|w_{(j_1,j_2)}|+1}\sum_{k=L}^{L+\ell_N-|w_{(j_1,j_2)}|}1_{[w_{(j_1,j_2)}]}(T^kx_i) \\ 
	&&\hspace{-1.5 in}=\frac{1}{\ell_N-|w_{(j_1,j_2)}|+1}\left(\sum_{k=0}^{L+\ell_N-|w_{(j_1,j_2)}|}1_{[w_{(j_1,j_2)}]}(T^kx_i)-\sum_{k=0}^{L-1}1_{[w_{(j_1,j_2)}]}(T^kx_i)\right) \\
	&&\hspace{-1.5 in}=\frac{L+\ell_N-|w_{(j_1,j_2)}|+1}{\ell_N-|w_{(j_1,j_2)}|+1}\cdot\frac{1}{L+\ell_N-|w_{(j_1,j_2)}|+1}\sum_{k=0}^{L+\ell_N-|w_{(j_1,j_2)}|}1_{[w_{(j_1,j_2)}]}(T^kx_i) \\ 
	&&\hspace{-1 in}-\frac{L}{\ell_N-|w_{(j_1,j_2)}|+1}\cdot\frac{1}{L}\sum_{k=0}^{L-1}1_{[w_{(j_1,j_2)}]}(T^kx_i). 
	\end{eqnarray*} 
But by~\eqref{eq:distance}, 
$$
\left|\frac{1}{L+\ell_N-|w_{(j_1,j_2)}|+1}\sum_{k=0}^{L+\ell_N-|w_{(j_1,j_2)}|}1_{[w_{(j_1,j_2)}]}(T^kx_i)-\mu_i([w_{(j_1,j_2)}])\right|<B\cdot\varepsilon 
$$
and since $L\geq M$,  
we have 
$$
\left|\frac{1}{L}\sum_{k=0}^{L-1}1_{[w_{(j_1,j_2)}]}(T^kx_i)-\mu_i([w_{(j_1,j_2)}])\right|<B\cdot\varepsilon. 
$$
Therefore 
	\begin{eqnarray*} 
	\left|F\left(u_i(L,\ell_N),w_{(j_1,j_2)}\right)-\mu_i([w_{(j_1,j_2)}])\right|&\hspace{3 in}& \\ 
	&&\hspace{-4.5 in}\leq\frac{L+\ell_N-|w_{(j_1,j_2)}|+1}{\ell_N-|w_{(j_1,j_2)}|+1}\cdot B\cdot\varepsilon+\frac{L}{\ell_N-|w_{(j_1,j_2)}|+1}\cdot B\cdot\varepsilon \\ 
	&&\hspace{-4.5 in}=\frac{2L+\ell_N-|w_{(j_1,j_2)}|+1}{\ell_N-|w_{(j_1,j_2)}|+1}\cdot B\cdot\varepsilon \\ 
	&&\hspace{-4.5 in}\leq\frac{2\lfloor(2-\delta)N\rfloor+\lfloor\delta N\rfloor-|w_{(j_1,j_2)}|+1}{\lfloor\delta N\rfloor-|w_{(j_1,j_2)}|+1}\cdot B\cdot\varepsilon. 
	\end{eqnarray*} 
By Definition~\eqref{eq:B}  that $B=\frac{\delta}{16-4\delta}$, for all sufficiently large $N$ this inequality implies 
\begin{equation}\label{eq:computation} 
\left|F\left(u_i(L,\ell_N),w_{(j_1,j_2)}\right)-\mu_i([w_{(j_1,j_2)}])\right|<\frac{\varepsilon}{2}. 
\end{equation} 
By~\eqref{eq:epsilon}, for all sufficiently large $N$ and all $L_1,L_2\in\{M,M+1,\dots,\lfloor(2-\delta)N\rfloor\}$ we have that if $1\leq i_1<i_2\leq d$, 
then the frequency with which $w_{(i_1,i_2)}$ occurs in $u_{i_1}(L_1,\ell_N)$ is different than its frequency in $u_{i_2}(L_2,\ell_N)$.  Therefore $u_{i_1}(L_1,\ell_N)\neq u_{i_2}(L_2,\ell_N)$.  For $1\leq i\leq d$ define 
$$
\mathcal{W}_i(N):=\left\{u_i(L,\ell_N)\colon M\leq L\leq\lfloor(2-\delta)N\rfloor\right\}\subseteq\mathcal{L}_{\ell_N}(X). 
$$ 
We have shown that for all sufficiently large $N$, 
if $1\leq i_1<i_2\leq d$, then
\begin{equation}\label{eq:words} 
\mathcal{W}_{i_1}(N)\cap\mathcal{W}_{i_2}(N)=\emptyset. 
\end{equation} 
Fix $i$ with $1\leq i\leq d$ and fix $N$ sufficiently large such that~\eqref{eq:words} holds.  If the words 
$$
u_i(N,M), u_i(N,M+1), u_i(N,M+2),\dots,u_i(N,\lfloor(1-\delta)N\rfloor)  
$$ 
are all distinct, then the set 
\begin{equation}
\label{eq:defSi} 
\mathcal{S}_i:=\{w\in\mathcal{L}_N(X)\colon\text{every subword of $w$ of length $\ell_N$ is an element of $\mathcal{W}_i(N)$}\} 
\end{equation}
contains at least $\lfloor(1-\delta)N\rfloor-M$ elements.  If, on the other hand, the words 
$$ 
u_i(N,M), u_i(N,M+1), u_i(N,M+2),\dots,u_i(N,\lfloor(1-\delta)N\rfloor) 
$$ 
are not all distinct, then there exist $M\leq L_1<L_2\leq \lfloor(1-\delta)N\rfloor$ such that $u_i(N,L_1)=u_i(N,L_2)$.  In this case, by Theorem~\ref{theorem:qz-revised} there exists $K\in\N$ such that 
	\begin{enumerate} 
	\item {\em (Distinct Words Condition)}: for all $K\leq k_1<k_2\leq K+N-\ell_N$ we have $u_i(N,k_1)\neq u_i(N,k_2)$; 
	\item {\em (Prefix First Occurrence Condition)}: for all $K\leq k\leq K+N-\ell_N$ there exists $\ell_N\leq l_k\leq N$ such that 
	$u_i(k,\ell_N)=u_i(l_k,\ell_N)$. 
	\end{enumerate} 
Thus in this case, the set 
\begin{equation}
\label{defTi}
\mathcal{T}_i:=\{w\in\mathcal{L}_N(X)\colon\text{the leftmost subword $w$ of length $\ell_N$ is an element of $\mathcal{W}_i(N)$}\} 
\end{equation}
contains at least $N-\ell_N$ elements. 

By~\eqref{eq:words}, $\mathcal{S}_{i_1}\cap\mathcal{S}_{i_2}=\emptyset$ whenever $i_1\neq i_2$ (and both sets are defined).  A similar statement holds when comparing any $\mathcal{S}_{i_1}$ to $\mathcal{T}_{i_2}$ for any $i_2$, or when comparing $\mathcal{T}_{i_1}$ to $\mathcal{T}_{i_2}$.  
Thus for each $1\leq i\leq d$, we have associated either the set $\mathcal{S}_i$ or the set $\mathcal{T}_i$ and 
$$ 
P_X(N)\geq d\cdot\min\{N-\ell_N, L_N-M\}=d\cdot\min\{N-\lfloor\delta N\rfloor, \lfloor(1-\delta)N\rfloor-M\}. 
$$ 
Therefore, 
$$ 
\frac{P_X(N)}{N}\geq\frac{d\cdot\min\{N-\lfloor\delta N\rfloor, \lfloor(1-\delta)N\rfloor-M\}}{N}, 
$$ 
which is larger than $d-2d\delta$ for all sufficiently large $N$, thus establishing~\eqref{eq:delta1}. 

To prove~\eqref{eq:delta2} , suppose that there exists $1\leq i\leq d$ such that the orbit closure of $x_i$ is not uniquely ergodic.  Then for any fixed $N\in\N$, there exist infinitely many $L\in\N$ such that $u_i(\ell_N,L)\notin\mathcal{W}_i(N)$.  Fix $N\in\N$.  

If the words 
$$ 
u_i(N,M), u_i(N,M+1), u_i(N,M+2),\dots,u_i(N,\lfloor(1-\delta)N\rfloor) 
$$ 
are all distinct, then we define $\mathcal{S}_i$ as in~\eqref{eq:defSi}.  
In this case, choose the smallest $L\geq M$ for which $u_i(\ell_N,L)\notin\mathcal{W}_i$; clearly $L>L_N$.  Then each of the words 
$$ 
u_i(N,L-N+\ell_N), u_i(N,L-N+\ell_N+1), \dots , u_i(N,L-\ell_N) 
$$ 
has the property that its leftmost subword of length $\ell_N$ is an element of $\mathcal{W}_i(N)$,  these words are pairwise distinct (in $u_i(N,L-N+\ell_N+j)$, 
and the leftmost occurrence of a subword of length $\ell_N$ that is not in $\mathcal{W}_i(N)$ begins at location $L-\ell_N-j$).  These $N-\ell_N$ words of length $N$ do not lie in $\mathcal{S}_i$, 
and are not  contained in any $\mathcal{S}_j$ or $\mathcal{T}_j$ for any $j\neq i$ (as defined in~\eqref{defTi}),  since their leftmost subword of length $\ell_N$ is in $\mathcal{W}_i$.  Therefore 
$$ 
P_X(N)\geq d\cdot\min\{N-\ell_N, L_N-M\}=d\cdot\min\{N-\lfloor\delta N\rfloor, \lfloor(1-\delta)N\rfloor-M\}+(N-\ell_N) 
$$ 
and so in this case, 
$$
\frac{P_X(N)}{N}\geq\frac{d\cdot\min\{N-\lfloor\delta N\rfloor, \lfloor(1-\delta)N\rfloor-M\}}{N}+\frac{N-\lfloor\delta N\rfloor}{N}.  
$$ 
If $N$ is sufficiently large, this is larger than $d+1-2d\delta$.  

Thus we are left with showing that there are infinitely many $N\in\N$ for which the words 
\begin{equation}
\label{eq:words2} 
u_i(N,M), u_i(N,M+1), u_i(N,M+2),\dots,u_i(N,\lfloor(1-\delta)N\rfloor) 
\end{equation}
are all distinct.  Fix some $N\in\N$ and assume that these words are not all distinct.  As before, let $L_1, L_2\in\{M,M+1,\dots,\lfloor(1-\delta)N\rfloor\}$ be distinct integers such that $u_i(N,L_1)=u_i(N,L_2)$.  Let $p$ be the minimal period of the word $u_i(N+L_2-L_1,L_1)$ and let $K$ be the largest integer for which $u_i(K,L_1)$ is periodic with period $p$ (note that $K$ is finite since $x_i$ is not eventually periodic).  Then the words 
$$ 
u_i(K,M), u_i(K,M+1), \dots, u_i(K,\lfloor(1-\delta)K\rfloor) 
$$ 
are all distinct: if $j>L_1-M$ then the word $u_i(K,M+j)$ begins with a word that is periodic of period $p$ and has length exactly $K-L_1-j$ (so no two words of this form can coincide), and if $j\leq L_1-M$ then $u_i(K,M+j)$ either begins with a word of length $K-L_1+j$ that is periodic of period $p$, or has a prefix of length at most $L_1$ followed by a word of length at least $K-L_1>N$ that is periodic of period $p$ (which occurs in a different location for each such $j$).  Therefore, for each $N\in\N$ there exists $K\geq N$ such that the words 
$$ 
u_i(K,M), u_i(K,M+1), \dots, u_i(K,\lfloor(1-\delta)K\rfloor) 
$$ 
are all distinct, and in particular there are infinitely many $N$ such that the words in~\eqref{eq:words2} are distinct.  This establishes~\eqref{eq:delta2}. 
\end{proof}

As  immediate corollaries of Theorem~\ref{theorem:main}, we have the theorems stated in the introduction: 
\begin{corollary*}[Theorem~\ref{th:upper}]
If $(X,\sigma)$ is a subshift and there exists $k\in\N$ such that 
$$
\liminf_{n\to\infty}\frac{P_X(n)}{n}<k, 
$$
then $(X,\sigma)$ has at most $k-1$ distinct, nonatomic, generic measures. 
\end{corollary*} 

\begin{corollary*}[Theorem~\ref{th:upper-sup}] 
If $(X,\sigma)$ is a subshift and there exists $k\in\N$ such that 
$$
\limsup_{n\to\infty}\frac{P_X(n)}{n}<k, 
$$ 
and if $(X,\sigma)$ has a generic measure $\mu$ and a generic point $x_{\mu}$ whose orbit closure is not uniquely ergodic, then $(X,\sigma)$ has at most $k-2$ distinct, nonatomic, generic measures. 
\end{corollary*} 

In Section~\ref{section:sharp}, we show that both of these corollaries are sharp. 
In particular, the linear growth rate in Theorem~\ref{th:upper} is optimal, in the sense that 
a superlinear growth rate does not suffice for showing that the set of ergodic measures is finite, 
and the technical condition of Theorem~\ref{th:upper-sup} (and in Theorem~\ref{theorem:main}) on the existence of a a point whose orbit closure is not uniquely ergodic can not be dropped.

\section{The natural coding of an IET}
\label{sec:IET}
Let $k\geq 1$ be an integer and $\pi$ be a permutation of $\{1, \ldots, k\}$.  Let $I=[0, \lambda]$ be an interval
and choose $0 = \lambda_0 < \lambda_1 < \ldots < \lambda_k = \lambda$.
The {\em interval exchange transformation} $T\colon [0,\lambda]\to [0,\lambda]$ is defined to be the map that is an
isometry on each subinterval $[\lambda_{i-1}, \lambda_i)$ for $i=1, \ldots, k$ and rearranges the order of these subintervals according to the permutation $\pi$.  We refer to this interval exchange transformation as a {\em $k$-IET} or just an {\em IET} when $k$ is clear from the context.  

Given an interval exchange transformation, there is a natural coding by an associated dynamical system.  For $x\in I$,
define ${\bf x } = (x_n)\in\{1, \ldots, k\}^\N$  by setting
$$
x_n = i \text{ if and only if } T^ix\in [\lambda_{i-1}, \lambda_i).
$$
The {\em language} of $\bf x$ is the set of all finite words that appear and the {\em natural coding} of the interval
exchange transformation is the symbolic system, endowed with the shift, that has the same language as $\bf x$.
The {\em natural symbolic cover} of an interval exchange transformation is the subshift that codes every $x\in I$,
meaning it is the symbolic system, endowed with the shift, whose language consists of all finite words that arise in the
orbit of any $x\in I$. 

If $T$ is a minimal interval exchange transformation, then any $x\in I$ gives rise to the same language and it suffices
to take the orbit of a single point.  More generally, the symbolic coding is not topologically conjugate to $T$, 
as up to countably many points may have multiple preimages (though a point can only have finitely many preimages).

We claim that a generic measure for an interval exchange transformation lifts to a generic measure
in the symbolic cover.
An open set in the symbolic cover is a cylinder set and thus corresponds to an interval or a finite
finite union of intervals  in $[0,\lambda]$.
Thus it suffices to check the claim for a finite interval $J\subseteq [0,\lambda]$.
Let $x\in[0,\lambda]$ be a generic point for the measure $\mu$.
Choose continuous functions $f$ and $g$ on $[0,\lambda]$ such that
$ 0\leq f\leq \one_J\leq g $
and
$\int g \,d\mu-\varepsilon/2 \leq \mu(J) \leq \int f \,d\mu+\varepsilon/2.
$
Then
$$
\left|\frac{1}{N}\sum_{n=1}^{N}f(T^nx) - \int f d\mu\right|<\varepsilon/2
$$
and the same holds for $g$.  Thus
$$\frac{1}{N}\sum_{n=0}^{N-1}\one_J(T^nx)
\leq  \frac{1}{N}\sum_{n=0}^{N-1}g(T^nx)
\leq \varepsilon/2 +  \int g \,d\mu
\leq \varepsilon + \int f \,d\mu
\leq \varepsilon + \frac{1}{N}\sum_{n=0}^{N-1}\one_J(T^nx).
$$
Thus the difference
$$\left|\mu(J) - \frac{1}{N}\sum_{n=0}^{N-1}\one_J(T^nx)\right| <\varepsilon.
$$
Since this holds for all $\varepsilon > 0$, for any open set $J\subset [0,\lambda]$, we have 
$$
\lim_{N\to\infty}\frac{1}{N}\sum_{n=0}^{N-1}\one_J(T^nx) = \mu(J).
$$

Write $\phi\colon (X,\sigma)\to ([0,\lambda], T)$ for the factor map from the symbolic coding $(X,\sigma)$ to 
the interval exchange $([0, \Lambda], T)$.  Let $\mathcal L(X)$ denote the language of the coding 
and let $\mu$ be a generic 
measure on $([0,\lambda], T)$ with generic point $x$.  Let $x^*\in\phi^{-1}(x)$.  Then for any 
word $w\in\mathcal L(X)$, 
$$
\lim_{N\to\infty}\frac{1}{N}
\sum_{n=0}^{N-1}\one_{[w]}(\sigma^nx^*) = 
\lim_{N\to\infty}\frac{1}{N}
\sum_{n=0}^{N-1}\one_{\phi([w])}(T^nx) = \mu(\phi([w])), 
$$
since $\phi([w])$ is a finite union of intervals.  Since $\mu$ is a nonatomic, generic measure, 
the pullback  $\phi^*(\mu([w])) = \phi^*(\mu(\phi^{-1}(\phi([w]))))$ is also nonatomic, as 
only countably many points in $([0, \lambda], T)$ have multiple pre-images and each of these only has 
finitely many preimages.  (In other words, the pushforward of the pullback of the measure is the measure itself.)  
Thus  a generic measure for the interval exchange transformation 
corresponds to a generic measure in the symbolic coding.

It is well known that an IET has linear complexity (see for example~\cite{FZ}).  We include a proof for completeness:  
\begin{proposition}
The natural coding of a minimal $k$-IET has complexity 
$$P(n) \leq (k-1)n+1.$$
If the $k$-IET satisfies the infinite distinct orbits condition (IDOC), then the complexity is exactly $P(n) = (k-1)n+1$.  
\end{proposition}

\begin{proof}
We proceed by induction on n.  For n=1, this is the alphabet $k$ and the result is clear.  
Assume that $P(n) \leq (k-1)n+1$.  Fixing a particular word of length $n$, 
the cylinder set defined by this word distinguishes an interval in the exchange, 
and by considering the cylinder sets associated to each word of length $n$, we obtain 
a partition of the exchange.  
Thus we have associated a partition $\CI$ of the exchange to the $(k-1)n+1$ words of length $n$, and this 
partition has $(k-1)n+2$ endpoints.  Furthermore, 
these endpoints all arise as iterates of the endpoints of the original $k+1$ 
endpoints of the interval exchange.  
Each of the $k+1$ original endpoints lies in some $T(I)$, where $T$ 
is the exchange map and $I$ is one of the intervals in the partition $\CI$. 
We note that if the exchange satisfies the IDOC condition, 
then the endpoints arise as distinct iterates, each of the original endpoints lies in the interior of 
some $T(I)$, 
but without this condition there may be overlap in the iterates and this is only an upper bound.

Thus we have $M \leq k-1$ intervals in $(T(I))_{I\in\CI}$ which cover all of the original endpoints.  
These $M$ intervals may each cover more than one of the original endpoints, say $m$ of them, and 
there are at most $m+1$ distinct ways to continue the orbit of a word of length $n$.  
Thus in total, we have $(k-1)n+1 - M + (k-1)+M$ continuations, which is exactly the bound $P(n+1) \leq kn+1$.  

If the exchange satisfies the IDOC condition, then as the endpoints arise as distinct iterates, we have 
that the complexity is exactly $P(n) = (k-1)n+1$.  
\end{proof}

Combining this with Theorem~\ref{theorem:main}, we have the statement of Theorem~\ref{thm:1.1}:
\begin{corollary*}[Theorem~\ref{thm:1.1}]
For $k>2$, a minimal $k$-IET has at most $k-2$ generic measures. 
\end{corollary*} 

\section{Sharpness}\label{section:sharp}
In this section show that the bound in Theorem~\ref{theorem:main} is sharp.  We recall the statement of Theorem~\ref{th:sharp} for convenience. 

\begin{theorem*}[Theorem~\ref{th:sharp}]
Let $d>1$ be fixed.  There exists a minimal subshift $(X,\sigma)$ such that
\begin{eqnarray*}
\liminf_{n\to\infty}\frac{P_X(n)}{n}&=&d, \\
\limsup_{n\to\infty}\frac{P_X(n)}{n}&=&d+1, 
\end{eqnarray*}
and $X$ has exactly $d$ ergodic measures.
\end{theorem*}

Before we delve into the details of the construction, we outline the basic ideas involved.  The ideas of this argument were partly inspired by a construction of a minimal and not uniquely ergodic subshift by Quas on mathoverflow~\cite{quas} (see also Denker, Grillenberger, and Sigmund~\cite{DGS}). 

Fixing $d> 1$ and the alphabet $\A = \{1,\ldots, d\}$, we inductively construct $d$ sequences of 
words $\{w_1^j\}_{j=1}^{\infty}$, $\{w_2^j\}_{j=1}^{\infty},\dots,\{w_d^j\}_{j=1}^{\infty}$ in $\CL(\A^\Z)$.  The procedure we use  constructs the words in these sequences in the following (somewhat unusual) order: $w_1^1, w_2^1, \ldots, w_d^1, w_1^2, w_2^2, \ldots, w_d^2, w_1^3, w_2^3,\ldots,w_d^3,\ldots$  That is, we first construct the first word in each of the sequences, then construct the second word in each of the sequences, and so on.  The words have the property that 
\begin{enumerate}
\item If $i_1, i_2\in\A$ and  $j_1<j_2$, then 
$w_{i_1}^{j_1}$ occurs as a subword of $w_{i_2}^{j_2}$ syndetically\footnote{Recall that a word occurs $v$ occurs syndetically in a word $w$ with gap $g$ if every subword of $w$ of length $g$ contains a copy of $v$ as a sub-subword.}, with gap size bounded by a constant that depends only on $j_1$; 
\item For any $i\in\A$ and $j\in\N$, the frequency with which the letter $i$ occurs in $w_i^j$ (as a percentage of the length of $w_i^j$) is greater than and absolute constant, greater than $1/2$.   
\end{enumerate}

By taking a limit along a subsequence of $\{w_1^j\}_{j=1}^{\infty}$, 
we  produce a semi-infinite word $w_1^{\infty}$ and taking its orbit closure under the shift $\sigma$ and the natural 
two sided extension, we obtain 
a closed subshift $X\subseteq\A^{\Z}$.  It  follows from the construction 
that $(X,\sigma)$ is minimal and that $w_i^j\in\mathcal{L}(X)$ for all $i\in\A$ and $j\in\N$.  
For fixed $i\in\A$, there are arbitrarily long words in $\mathcal{L}(X)$ for which the frequency 
of letter $i$ is greater than (a constant greater than) $1/2$ and so the system 
$(X,\sigma)$ has an ergodic measure assigning the cylinder set $[i]$ measure larger than $1/2$.  
Thus $(X,\sigma)$ has at least $|\A|=d$ ergodic measures.  By carefully choosing the lengths of 
the words, we further show that the system $(X,\sigma)$ satisfies the upper and lower bounds 
on the complexity as in the statement of the theorem.  
Applying Theorem~\ref{theorem:main}, it follows that $(X,\sigma)$ has at most $d$ ergodic measures, and so exactly 
$d$ ergodic measures. 

We now make these ideas precise: 
\begin{proof}[Proof of Theorem~\ref{th:sharp}]
Let $\A:=\{1,2,\dots,d\}$.  Choose $\kappa_1,\kappa_2,\dots$ to be a sequence of real numbers in $(0,1)$ such that  
$$
\prod_{j=1}^{\infty}\kappa_j>1/2
$$
choose $\delta_1,\delta_2,\dots$ to be a strictly decreasing sequence of real numbers in $(0,1)$ such that $\lim\delta_j=0$.  

\subsection*{Step 1} ({\em Construction of the sequences $\{w_1^j\}_{j=1}^{\infty},\dots,\{w_d^j\}_{j=1}^{\infty}$}): Define the word 
$$
w_1^1:=\underbrace{11\cdots1}_{\text{length }N_{(1,1)}^{[1]}}\hspace{-0.1 in}234\cdots d 
$$
where $N_{(1,1)}^{[1]}\in\N$ is chosen such that $N_{(1,1)}^{[1]}>\kappa_1|w_1^1|$.  Next define the word 
$$
w_2^1:=\underbrace{w_1^1w_1^1\cdots w_1^1}_{\text{length }N_{(2,1)}^{[1]}}\underbrace{222\cdots2}_{\text{length }N_{(2,2)}^{[1]}}\underbrace{333\cdots3}_{\text{length }N_{(2,3)}^{[1]}}\cdots\underbrace{ddd\cdots d}_{\text{length }N_{(2,d)}^{[1]}} 
$$
where $N_{(2,1)}^{[1]},\dots,N_{(2,d)}^{[1]}\in\N$ are chosen such that 
\begin{equation}\label{eq:control1}
\begin{aligned}
|w_1^1| & <(\delta_1)^2\cdot N_{(2,1)}^{[1]}<(\delta_1)^4\cdot N_{(2,d)}^{[1]}<(\delta_1)^6\cdot N_{(2,d-1)}^{[1]}<(\delta_1)^8\cdot N_{(2,d-2)}^{[1]} \\
& <\cdots<(\delta_1)^{2d-2}\cdot N_{(2,3)}^{[1]}<(\delta_1)^{2d}\cdot N_{(2,2)}^{[1]}
\end{aligned}
\end{equation}
and $N_{(2,2)}^{[1]}>\kappa_1 |w_2^1|$.  Note that the ordering of the lengths $N_{(2,k)}^{[1]}$ is important, 
with the index $k$ cyclically passing from $1$ to $d$ to $d-1$ and down to $2$.  
This choice of the lengths is used only 
in estimating the growth of $P_X(n)$; the exact choices of the lengths and the 
estimates of~\eqref{eq:control1} can be ignored for a first reading of Steps 1 and 2 of this construction.  

For $i<d$, inductively define the word 
\begin{multline*}
w_{i+1}^1:= \\ 
\underbrace{w_1^1w_1^1\cdots w_1^1}_{\text{length }N_{(i+1,1)}^{[1]}}\underbrace{w_2^1w_2^1\cdots w_2^1}_{\text{length }N_{(i+1,2)}^{[1]}}\cdots\underbrace{w_i^1w_i^1\cdots w_i^1}_{\text{length }N_{(i+1,i)}^{[1]}}\underbrace{(i+1)(i+1)\cdots(i+1)}_{\text{length }N_{(i+1,i+1)}^{[1]}}\cdots\underbrace{ddd\cdots d}_{\text{length }N_{(i+1,d)}^{[1]}}
\end{multline*} 
where $N_{(i+1,1)}^{[1]},\dots,N_{(i+1,d)}^{[1]}\in\N$ are chosen such that 
\begin{equation}\label{eq:control2}
\begin{aligned}
|w_i^1| & <(\delta_1)^2\cdot N_{(i+1,i)}^{[1]}<(\delta_1)^4\cdot N_{(i+1,i-1)}^{[1]}<(\delta_1)^6\cdot N_{(i+1,i-2)}^{[1]} \\
& <(\delta_1)^8\cdot N_{(i+1,i-3)}^{[1]}<\cdots<(\delta_1)^{2i}\cdot N_{(i+1,1)}^{[1]}<(\delta_1)^{2i+2}\cdot N_{(i+1,d)}^{[1]} \\
& <(\delta_1)^{2i+4}\cdot N_{(i+1,d-1)}^{[1]}<\cdots<(\delta_1)^{2d-2}\cdot N_{(i+1,i+2)}^{[1]}<(\delta_1)^{2d}\cdot N_{(i+1,i+1)}^{[1]}
\end{aligned}
\end{equation}
and $N_{(i+1,i+1)}^{[1]}>\kappa_1|w_{i+1}^1|$.   Again, the lengths are chosen such that we are able to 
control the growth of the complexity, and the index $k$ in $N_{(i+1,k)}^{[1]}$ is taken in a cyclical order.  

For each $i\in\mathcal{\A}$, it follows immediately from the construction that: 
	\begin{enumerate}[(a)] 
	\item \label{eq:a}
	Every letter in $\mathcal{A}$ appears in $w_i^1$; 
	\item \label{eq:b}
	The frequency with which the letter $i$ occurs in $w_i^1$ is at least $\kappa_1$. 
	\end{enumerate} 
	
We continue to define the words inductively.  Assuming that 
we have already defined words $w_1^j, w_2^j,\dots, w_d^j$, we define
$$
w_1^{j+1}:=\underbrace{w_1^jw_1^j\cdots w_1^j}_{\text{length }N_{(1,1)}^{[j]}}\underbrace{w_2^jw_2^j\cdots w_2^j}_{\text{length }N_{(1,2)}^{[j]}}\cdots\underbrace{w_d^jw_d^j\cdots w_d^j}_{\text{length }N_{(1,d)}^{[j]}} 
$$
where 
\begin{equation}\label{eq:control3}
\begin{aligned}
|w_d^j| & <(\delta_{j+1})^2\cdot N_{(1,d)}^{[j+1]}<(\delta_{j+1})^4\cdot N_{(1,d-1)}^{[j+1]}<(\delta_{j+1})^6\cdot N_{(1,d-2)}^{[j+1]} \\
& <(\delta_{j+1})^8\cdot N_{(1,d-3)}^{[j+1]}<\cdots<(\delta_1)^{2d-2}\cdot N_{(1,2)}^{[j+1]}<(\delta_1)^{2d}\cdot N_{(1,1)}^{[j+1]} 
\end{aligned} 
\end{equation} 
and $N_{(1,1)}^{[j+1]}>\kappa_{j+1} |w_1^{j+1}|$.  
We have analogs of properties~\eqref{eq:a} and~\eqref{eq:b} for the base case of the construction: 
each of the words $w_1^j, w_2^j,\dots,w_d^j$ occurs as a subword of $w_1^{j+1}$ and the frequency with which the letter $1$ occurs in $w_1^{j+1}$ is at least $\prod_{k=1}^{j+1}\kappa_k$, provided that the frequency with which it occurs in $w_1^j$ was at least $\prod_{k=1}^j\kappa_k$. 

Continuing inductively, for $i<d$, we define the word (note the change in superscript half way through) 
\begin{multline*}
w_{i+1}^{j+1}:=\\ 
\underbrace{w_1^{j+1}w_1^{j+1}\cdots w_1^{j+1}}_{\text{length }N_{(i+1,1)}^{[j+1]}}\underbrace{w_2^{j+1}\cdots w_2^{j+1}}_{\text{length }N_{(i+1,2)}^{[j+1]}}\cdots\underbrace{w_i^{j+1}\cdots w_i^{j+1}}_{\text{length }N_{(i+1,i)}^{[j+1]}}\underbrace{w_{i+1}^j\cdots w_{i+1}^j}_{\text{length }N_{(i+1,i+1)}^{[j+1]}}\cdots\underbrace{w_d^j\cdots w_d^j}_{\text{length }N_{(i+1,d)}^{[j+1]}} 
\end{multline*} 
where 
\begin{equation}\label{eq:control4}
\begin{aligned}
|w_i^{j+1}| & <(\delta_{j+1})^2\cdot N_{(i+1,i)}^{[j+1]}<(\delta_{j+1})^4\cdot N_{(i+1,i-1)}^{[j+1]}<(\delta_{j+1})^6\cdot N_{(i+1,i-2)}^{[j+1]} \\
& <(\delta_{j+1})^8\cdot N_{(i+1,i-3)}^{[j+1]}<\cdots<(\delta_{j+1})^{2i}\cdot N_{(i+1,1)}^{[j+1]} \\ 
& <(\delta_{j+1})^{2i+2}\cdot N_{(i+1,d)}^{[j+1]}<(\delta_{j+1})^{2i+4}\cdot N_{(i+1,d-1)}^{[j+1]}<\cdots \\ 
& <(\delta_{j+1})^{2d-2}\cdot N_{(i+1,i+2)}^{[j+1]}<(\delta_{j+1})^{2d}\cdot N_{(i+1,i+1)}^{[j+1]} 
\end{aligned} 
\end{equation}  
and $N_{(i+1,i+1)}^{[j+1]}>\kappa_{j+1}|w_{i+1}^{j+1}|$.  Again, we point out that the words $w_1^j, w_2^j,\dots,w_d^j$ occur as subwords of $w_{i+1}^{j+1}$, and the frequency with which the letter $i+1$ occurs in $w_{i+1}^{j+1}$ is at least $\prod_{k=1}^{j+1}\kappa_k$, provided that the frequency with which it occurs in $w_{i+1}^j$ was at least $\prod_{k=1}^j\kappa_k$. 

By induction, we obtain sequences $\{w_1^j\}_{j=1}^{\infty}, \{w_2^j\}_{j=1}^{\infty},\dots,\{w_d^j\}_{j=1}^{\infty}$ satisfying: 
	\begin{enumerate}[(a)] 
	\item For any $j>2$, any $1\leq k<j-1$, and any $i_1, i_2\in\A$, the word $w_{i_1}^k$ occurs in each of the words $w_1^{k+1}, w_2^{k+1},\dots, w_d^{k+1}$ and therefore occurs in $w_{i_2}^j$ (which may be written as a concatenation of these words) syndetically, and the maximal gap length is at most 
	$$
	g_k:=\max\{|w_l^{k+1}|\colon l\in\A\};  
	$$
	\item For any $i\in\A$ and any $j\in\N$, the frequency with which the letter $i$ occurs as a subword of $w_i^j$ is at least $\prod_{k=1}^j\kappa_k\geq\prod_{k=1}^{\infty}\kappa_k>1/2$. 
	\end{enumerate} 
	
We further note that given the freedom in which the lengths are chosen, we 
can assume that $N_{(i,k)}^{[j]}$ divides $N_{(i,k)}^{[j+1]}$ for all $i,k\in\A$ and all $j\in\N$.  We make this assumption for the remainder of the proof. 

\subsection*{Step 2} ({\em Construction and ergodic properties of the subshift $(X,\sigma)$}): Observe that $w_1^j$ is the leftmost subword of $w_1^{j+1}$ for all $j\in\N$, and so we can define a (one-sided) infinite word $w_1^{\infty}$ by declaring that for all $j$, the leftmost subword of $w_1^{\infty}$ of length $|w_1^j|$ is $w_1^j$.  Then for any $i\in\A$ and any $j\in\N$, the word $w_i^j$ occurs as a subword of $w_1^{\infty}$ syndetically.  Moreover, every subword of $w_1^{\infty}$ occurs as a sub-subword of $w_1^j$ for some $j$.  Therefore all subwords of $w_1^{\infty}$ occur syndetically.  

Let $X\subset\A^{\Z}$ be the set of all bi-infinite sequences whose language is comprised only of 
subwords of $w_1^{\infty}$, 
meaning it is the natural extension of the closure of $w_1^{\infty}$ under $\sigma$.  Then $(X,\sigma)$ is minimal and  $w_i^j\in\mathcal{L}(X)$ for all $i\in\A$ and $j\in\N$.  Therefore, for fixed $i\in\A$, there are arbitrarily long words in $\mathcal{L}(X)$ for which the frequency with which the letter $i$ occurs is at least $\prod_{k=1}^{\infty}\kappa_k>1/2$.  Consequently, there exists an ergodic measure $\mu_i$ supported on $X$ for which $\mu_i([i])>1/2$.  It follows that $\mu_i([j])<1/2$ for all $j\neq i$ and so $\mu_j\neq\mu_i$ for any $j\neq i$.  Thus $(X,\sigma)$ has at least $d$ ergodic measures. 
 If we can show that 
$$
\liminf_{n\to\infty}\frac{P_X(n)}{n}<d+1, 
$$
then there are at most $d$ ergodic measures by Theorem~\ref{theorem:main}; hence exactly $d$.  So it remains only to show:  
\begin{eqnarray*} 
\liminf_{n\to\infty}\frac{P_X(n)}{n}&=&d, \\
\limsup_{n\to\infty}\frac{P_X(n)}{n}&=&d+1. 
\end{eqnarray*} 

\medskip

\subsection*{Step 3} ({\em Analysis of the growth rate of $P_X(n)$}): Let $n>|w_1^2|$ be a fixed integer.  We estimate the number of words in $\mathcal{L}_n(X)$ (recall that this number is, by definition, $P_X(n)$).  By construction, 
$$
|w_1^1|<|w_2^1|<\cdots<|w_d^1|<|w_1^2|<|w_2^2|<\cdots<|w_d^2|<|w_1^3|<\cdots 
$$ 
We make the convention that $w_{d+1}^j:=w_1^{j+1}$, $w_{d+2}^j:=w_2^{j+1}$, and so on (with the analogous convention for $N_{(i_1,i_2)}^{[j]}$ when $i_2>d$).  Therefore, there exist $i_1\in\A$ and $j_1\in\N$ such that 
$$ 
|w_{i_1}^{j_1}|\leq n<|w_{i_1+1}^{j_1}|. 
$$ 

With this convention, observe that 
\begin{multline*} 
w_{i_1+1}^{j_1+1}:= \\ 
\underbrace{w_1^{j_1+1}\cdots w_1^{j_1+1}}_{\text{length }N_{(i_1+1,1)}^{[j_1+1]}}\underbrace{w_2^{j_1+1}\cdots w_2^{j_1+1}}_{\text{length }N_{(i_1+1,2)}^{[j_1+1]}}\cdots\underbrace{w_{i_1}^{j_1+1}\cdots w_{i_1}^{j_1+1}}_{\text{length }N_{(i_1+1,i_1)}^{[j_1+1]}}\underbrace{w_{i_1+1}^{j_1}\cdots w_{i_1+1}^{j_1}}_{\text{length }N_{(i_1+1,i_1+1)}^{[j_1]}}\cdots\underbrace{w_d^{j_1}\cdots w_d^{j_1}}_{\text{length }N_{(i_1+1,d)}^{[j_1]}}, 
\end{multline*} 
where 
\begin{equation}\label{eq:size} 
n<|w_{i_1+1}^{j_1}|<|w_{i_1+2}^{j_1}|<\cdots<|w_d^{j_1}|<|w_1^{j_1+1}|<\cdots<|w_{i_1}^{j_1+1}|. 
\end{equation} 
It follows from the construction that if $i_2\in\A$ and $j_2\in\N$ is such that $|w_{i_2}^{j_2}|\geq|w_{i_1+1}^{j_1+1}|$, then $w_{i_2}^{j_2}$ can also be written as a concatenation of words from the set 
$$\{w_1^{j_1+1},w_2^{j_1+1},\dots,w_{i_1}^{j_1+1},w_{i_1+1}^{j_1}, w_{i_1+2}^{j_1},\dots,w_d^{j_1}\}  = \{w_{i_1}^{j_1}, w_{i_1+1}^{j_1}, w_{i_1+2}^{j_1},\dots,w_{i_1+d-1}^{j_1}\}. 
$$
Moreover, there are restrictions on the order in which these words may be concatenated in $w_{i_2}^{j_2}$: 
	\begin{enumerate} 
	\item If $i_1+1\leq i<i_1+d$, then the only words that may be concatenated with $w_i^{j_1}$ are $w_i^{j_1}$ itself and 
	$w_{i+1}^{j_1}$; 
	\item The only words that may be concatenated with $w_{i_1+d}^{j_1}(=w_{i_1}^{j_1+1})$ are $w_{i_1}^{j_1+1}$ itself and $w_{i_1+1}^{j_1}$. 
	\end{enumerate} 
Therefore, by~\eqref{eq:size}, the only words of length $n$ that appear as subwords of $w_{i_2}^{j_2}$ are those which appear as subwords of words from the set: 
\begin{equation}\label{eq:set} 
\left\{w_i^{j_1}w_i^{j_1}\colon i_1<i\leq i_1+d\right\}\cup\left\{w_i^{j_1}w_{i+1}^{j_1}\colon i_1<i<i_1+d\right\}\cup\left\{w_{i_1+d}^{j_1}w_{i_1+1}^{j_1}\right\} , 
\end{equation} 
with superscripts following the convention that if the subscript is larger than $d$, increment the superscript by $1$.  
Since all words in $\mathcal{L}_n(X)$ occur as subwords of $w_1^{j_2}$ for all sufficiently large $j_2$, we have that all words in $\mathcal{L}_n(X)$ appear as subwords of the $2d$ words in the set in~\eqref{eq:set}. 

We now analyze the words that appear in~\eqref{eq:set} by decomposing them into words of length comparable to $n$.  By construction, if $k\geq1$ then $w_{i_1+k}^{j_1}$ can be written as a concatenation of words from the set (recall the divisibility of the lengths assumed at the end of Step 1) 
\begin{equation}\label{eq:set3} 
\left\{w_{i_1+1}^{j_1},\underbrace{w_{i_1+2}^{j_1-1}\cdots w_{i_1+2}^{j_1-1}}_{N_{(i_1+2,i_1+2)}^{[j_1]}},\dots,\underbrace{w_d^{j_1-1}\cdots w_d^{j_1-1}}_{N_{(i_1+2,d)}^{[j_1]}},\underbrace{w_1^{j_1}\cdots w_1^{j_1}}_{N_{(i_1+2,1)}^{[j_1]}},\dots,\underbrace{w_{i_1}^{j_1}\cdots w_{i_1}^{j_1}}_{N_{(i_1+2,i_1)}^{[j_1]}}\right\} 
\end{equation} 
obeying the analogous rules for concatenation (a word may concatenate with itself or with the word whose subscript is one larger, understood cyclically).  Moreover, 
\begin{equation}\label{eq:set2} 
|w_{i_1+2}^{j_1-1}|<|w_{i_1+3}^{j_1-1}|<\cdots<|w_d^{j_1-1}|<|w_1^{j_1}|<\cdots<|w_{i_1}^{j_1}|\leq n<|w_{i_1+1}^{j_1}| 
\end{equation} 
and 
\begin{equation}\label{eq:condition} 
	n<|w_{i_1+1}^{j_1}|<N_{(i_1+2,m)}^{[j_1]} 
\end{equation} 
for all $1\leq m\leq d$ (again, if $i_1+2>d$ then increment the superscript of $N_{(i_1+2,m)}^{[j_1]}$ by one and reduce the subscript by $d$).  In particular, every word in the set~\eqref{eq:set} can be obtained by concatenating words from the set~\eqref{eq:set3}. 

For $i_1+2\leq i<i_1+d+1$, define 
$$ 
p_i:=\cdots w_i^{j_1-1}w_i^{j_1-1}w_i^{j_1-1}w_{i+1}^{j_1-1}w_{i+1}^{j_1-1}w_{i+1}^{j_1-1}\cdots 
$$ 
to the the bi-infinite word whose restriction to to the set $\{n\geq0\}$ is an infinite concatenation of the word $w_{i+1}^{j_1-1}$ with itself, and whose restriction to the set $\{n<0\}$ is an infinite concatenation of the word $w_i^{j_1-1}$ with itself.  Similarly define 
$$ 
p_{i_1+d+1}:=\cdots w_{i_1+1}^{j_1}w_{i_1+1}^{j_1}w_{i_1+1}^{j_1}w_{i_1+2}^{j_1-1}w_{i_1+2}^{j_1-1}w_{i_1+2}^{j_1-1}\cdots 
$$ 
The set of words length $n$ that arise by concatenating words from the set~\eqref{eq:set3} is precisely the set of words of length $n$ that appear in $p_{i_1+1}, p_{i_1+2},\dots,p_{i_1+d}$, by~\eqref{eq:condition}.  By the estimates in~\eqref{eq:control1},~\eqref{eq:control2},~\eqref{eq:control3}, and~\eqref{eq:control4}, we have that 
$$ 
|w_i^{j_1-1}|<\delta_{j_1}\cdot|w_{i_1}^{j_1}|\leq\delta_{j_1}\cdot n 
$$ 
for all $i_1+1<i<d+i_1$.  It follows that:
	\begin{enumerate} 
	\item If $i_1+1<i<i_1+d-1$, then the number of factors of $p_i$ of length $n$ is at least $n+1$ (since $p_i$ is aperiodic) and 
	at most $n+2\delta_{j_1}n$ (there are at most $\delta_{j_1}n$ factors in each ``periodic part'' of $p_i$ and at most $n$ 
	transitional factors obtained from words that overlap the origin); 
	\item The number of factors of $p_{i_1+d-1}$ of length $n$ is at least $n+1$ and at most $n+\delta_{j_1}n+|w_{i_1}^{j_1}|$; 
	\item The only new factors of $p_{i_1+d}$ are the $n+1$ transitional factors which appear in 
	$\underbrace{w_{i_1}^{j_1}\cdots w_{i_1}^{j_1}}_{N_{(i_1+2,i_1)}^{[j_1]}}w_{i_1+1}^{j_1}$ as well as factors that appear in  
	$w_{i_1+1}^{j_1}w_{i_1+1}^{j_1}$; 
	\item The only new factors of $p_{i_1+1}$ are the $n+1$ transitional factors which appear in $w_{i_1+1}^{j_1}\underbrace{w_{i_1+2}^{j_1-1}\cdots w_{i_1+2}^{j_1-1}}_{N_{(i_1+2,i_1+2)}^{[j_1]}}$.  
	\end{enumerate} 
Thus we are left with counting subwords of $w_{i_1+1}^{j_1}w_{i_1+1}^{j_1}$ that have not already appeared.

Write 
\begin{multline*}
w_{i_1+1}^{j_1}:= \\
\underbrace{w_1^{j_1}\cdots w_1^{j_1}}_{\text{length }N_{(i_1+1,1)}^{[j_1]}}\underbrace{w_2^{j_1}\cdots w_2^{j_1}}_{\text{length }N_{(i_1+1,2)}^{[j_1]}}\cdots\underbrace{w_{i_1}^{j_1}\cdots w_{i_1}^{j_1}}_{\text{length }N_{(i_1+1,i_1)}^{[j_1]}}\underbrace{w_{i_1+1}^{j_1-1}\cdots w_{i_1+1}^{j_1-1}}_{\text{length }N_{(i_1+1,i_1+1)}^{[j_1]}}\cdots\underbrace{w_d^{j_1-1}\cdots w_d^{j_1-1}}_{\text{length }N_{(i_1+1,d)}^{[j_1]}}.  
\end{multline*}
Then by~\eqref{eq:control4} and the observation that $|w_{i_1+1}^{j_1}|<N_{(i_1+1,i_1+1)}^{[j_1]}/\delta_{j_1}$ for all sufficiently large $j_1$, we have 
\begin{equation*}
\begin{aligned}
|w_{i_1}^{j_1}| & <\delta_{j_2}\cdot N_{(i_1+1,i_1)}^{[j_1]}<N_{(i_1+1,i_1)}^{[j_1]}/\delta_{j_1}<\delta_{j_1}\cdot N_{(i_1+1,i_1-1)}^{[j_1]} <N_{(i_1+1,i_1-1)}^{[j_1]}/\delta_{j_1} \\ 
& <\delta_{j_1}\cdot N_{(i_1+1,i_1-2)}^{[j_1]}<N_{(i_1+1,i_1-2)}^{[j_1]}/\delta_{j_1}<\delta_{j_1}\cdot N_{(i_1+1,i_1-3)}^{[j_1]}<N_{(i_1+1,i_1-3)}^{[j_1]}/\delta_{j_1} \\ 
& <\cdots<\delta_{j_1}\cdot N_{(i_1+1,1)}^{[j_1]}<N_{(i_1+1,1)}^{[j_1]}/\delta_{j_1}<\delta_{j_1}\cdot N_{(i_1+1,d)}^{[j_1]}<N_{(i_1+1,d)}^{[j_1]}/\delta_{j_1} \\ 
& <\delta_{j_1}\cdot N_{(i_1+1,d-1)}^{[j_1]}<N_{(i_1+1,d-1)}^{[j_1]}/\delta_{j_1}<\cdots<\delta_{j_1}\cdot N_{(i_1+1,i_1+2)}^{[j_1]} \\ 
& <N_{(i_1+1,i_1+2)}^{[j_1]}/\delta_{j_1}<\delta_{j_1}\cdot N_{(i_1+1,i_1+1)}^{[j_1]}<|w_{i_1+1}^{j_1}|<N_{(i_1+1,i_1+1)}^{[j_1]}/\delta_{j_1}. 
\end{aligned} 
\end{equation*} 
Thus there are four possibilities: 
	\begin{enumerate}
	\item $|w_{i_1}^{j_1}|\leq n<N_{(i_1+1,i_1)}^{[j_1]}$; 
	\item $N_{(i_1+1,i_1)}^{[j_1]}\leq n<N_{(i_1+1,i_1-1)}^{[j_1]}$ (indices taken modulo $d$);
	\item $N_{(i_1+1,i_1-1)}^{[j_1]}\leq n<N_{(i_1+1,i_1+1)}^{[j_1]}$ and there exists $i_2\in\A\setminus\{i_1,i_1+1\}$ such that 
	$$ 
	N_{(i_1+1,i_2)}^{[j_1]}\leq n<N_{(i_1+1,i_2-1)}^{[j_1]}\text{ (indices taken modulo $d$)}; 
	$$ 
	\item $n>N_{(i_1+1,i_1+1)}^{[j_1]}$. 
	\end{enumerate}
	
In case (i), there are no words of length $n$ in $w_{i_1+1}^{j_1}w_{i_1+1}^{j_1}$ that were not previously counted (all blocks in its decomposition are of length larger than $n$).  In this case we have the estimate 
\begin{equation}\label{eq:argument} 
P_X(n)\leq (d-1)\delta_{j_1}n+|w_{i_1}^{j_1}|+dn. 
\end{equation} 
In particular, since $|w_{i_1}^{j_1}|<\delta_{j_1}N_{(i_1+1,i_1)}^{[j_1]}$, we are in case (i) when $n=\left\lfloor\frac{|w_{i_1}^{j_1}|}{\delta_{j_1}}\right\rfloor$ and so equation~\eqref{eq:argument} holds.  This implies that 
$$ 
P_X\left(\left\lfloor\frac{|w_{i_1}^{j_1}|}{\delta_{j_1}}\right\rfloor\right)\leq (d+d\delta_{j_1})\cdot\left\lfloor\frac{|w_{i_1}^{j_1}|}{\delta_{j_1}}\right\rfloor. 
$$ 
This situation arises infinitely often (once for each $\delta_j$) and since $\delta_j\xrightarrow{j\to\infty}0$, 
$$ 
\liminf_{n\to\infty}\frac{P_X(n)}{n}\leq d. 
$$ 
Combining this with the fact that $(X,\sigma)$ has at least $d$ distinct nonatomic ergodic measures and applying Theorem~\ref{theorem:main}, we have that 
$$ 
\liminf_{n\to\infty}\frac{P_X(n)}{n}=d. 
$$ 
In particular, this implies that there are exactly $d$ ergodic measures.  

In case (ii),  we have 
$$ 
N_{(i_1+1,i_1)}^{[j_1]}\leq n<N_{(i_1+1,i_1-1)}^{[j_1]} 
$$ 
and $n<N_{(i_1+1,k)}^{[j_1]}$ for all $k\in\A\setminus\{i_1\}$.  In this case, the only new words of length $n$ that arise in $w_{i_1+1}^{j_1}w_{i_1+1}^{j_1}$ are the $n-N_{(i_1+1,i_1)}^{[j_1]}$ transitional words that arise in 
$$ 
\underbrace{w_{i_1-1}^{j_1}\cdots w_{i_1-1}^{j_1}}_{\text{length }N_{(i_1+1,i_1-1)}^{[j_1]}}\underbrace{w_{i_1}^{j_1}\cdots w_{i_1}^{j_1}}_{\text{length }N_{(i_1+1,i_1)}^{[j_1]}}\underbrace{w_{i_1+1}^{j_1-1}\cdots w_{i_1+1}^{j_1-1}}_{\text{length }N_{(i_1+1,i_1+1)}^{[j_1]}}, 
$$ 
where a word is transitional if it completely contains the middle block (all other blocks have length larger than $n$ and so contribute no new words).  Thus, in case (ii), 
\begin{equation}\label{eq:condition3} 
P_X(n)\leq dn+2d\delta_{j_1}n+\left(n-N_{(i_1+1,i_1)}^{[j_1]}\right)\leq(d+1)n+2d\delta_{j_1}. 
\end{equation} 

In case (iii), 
\begin{equation}\label{eq:estimate-a} 
n\geq N_{(i_1+1,i_2)}^{[j_1]}>\delta_{j_1}\cdot N_{(i_1+1,i_2+1)}^{[j_1]}>\delta_{j_1}\cdot N_{(i_1+1,i_2+2)}^{[j_1]}>\cdots>\delta_{j_1}\cdot N_{(i_1+1,i_1)}^{[j_1]} 
\end{equation} 
and 
$$ 
n<N_{(i_1+1,i_2-1)}^{[j_1]}<N_{(i_1+1,i_2-2)}^{[j_1]}<\cdots<N_{(i_1+1,i_1+1)}^{[j_1]} 
$$ 
by~\eqref{eq:control1},~\eqref{eq:control2},~\eqref{eq:control3}, and~\eqref{eq:control4}.  Therefore the only new words of length $n$ that arise in $w_{i_1+1}^{j_1}w_{i_1+1}^{j_1}$ are the transitional words that arise in 
$$ 
\underbrace{w_{i_2-1}^{j_1}\cdots w_{i_2-1}^{j_1}}_{\text{length }N_{(i_1+1,i_2-1)}^{[j_1]}}\underbrace{w_{i_2}^{j_1}\cdots w_{i_2}^{j_1}}_{\text{length }N_{(i_1+1,i_2)}^{[j_1]}}\cdots\underbrace{w_{i_1}^{j_1}\cdots w_{i_1}^{j_1}}_{\text{length }N_{(i_1+1,i_1)}^{[j_1]}}\underbrace{w_{i_1+1}^{j_1-1}\cdots w_{i_1+1}^{j_1-1}}_{\text{length }N_{(i_1+1,i_1+1)}^{[j_1]}}. 
$$ 
That is, this word decomposes into blocks the first and last of which have length larger than $n$; the transitional words are those that fully contain one of the blocks of length smaller than $n$.  There are at most 
$$ 
\left(n-N_{(i_1+1,i_2)}^{[j_1]}\right)+N_{(i_1+1,i_2+1)}^{[j_1]}+N_{(i_1+1,i_2+2)}^{[j_1]}+\cdots+N_{(i_1+1,i_1-1)}^{[j_1]}+\left(n-N_{(i_1+1,i_1)}^{[j_1]}\right) 
$$ 
such blocks.  By~\eqref{eq:estimate-a}, this is at most $n-N_{(i_1+1,i_2)}^{[j_1]}+d\delta_{j_1}n$.  So in case (iii), 
\begin{equation}\label{eq:calculation1}
P_X(n)\leq dn+2d\delta_{j_1}n+n-N_{(i_1+1,i_2)}^{[j_1]}+d\delta_{j_1}n\leq(d+1)n+3d\delta_{j_1}n. 
\end{equation} 

Finally, in case (iv), the only new words are the transitional words that occur in 
$$ 
\underbrace{w_{i_1+1}^{j_1-1}\cdots w_{i_1+1}^{j_1-1}}_{\text{length }N_{(i_1+1,i_1+1)}^{[j_1]}}\underbrace{w_{i_1+2}^{j_1-1}\cdots w_{i_1+2}^{j_1-1}}_{\text{length }N_{(i_1+1,i_1+2)}^{[j_1]}}\cdots\underbrace{w_{i_1}^{j_1}\cdots w_{i_1}^{j_1}}_{\text{length }N_{(i_1+1,i_1)}^{[j_1]}}\underbrace{w_{i_1+1}^{j_1-1}\cdots w_{i_1+1}^{j_1-1}}_{\text{length }N_{(i_1+1,i_1+1)}^{[j_1]}}, 
$$ 
where now a word is transitional if it completely contains any of the blocks between the first and the last.  However, by~\eqref{eq:control1},~\eqref{eq:control2},~\eqref{eq:control3}, and~\eqref{eq:control4}, we have 
$$ 
\delta_{j_1}n\geq\delta_{j_1}N_{(i_1+1,i_1+1)}^{[j_1]}>N_{(i_1+1,k)}^{[j_1]} 
$$ 
for all $k\in\A\setminus\{i_1+1\}$, and so there are at most $n+d\delta_{j_1}$ such words.  Thus for case (iv), we have 
\begin{equation}\label{eq:calculation3} 
P_X(n)\leq dn+2d\delta_{j_1}n+n+d\delta_{j_1}=(d+1)n+3d\delta_{j_1}n. 
\end{equation} 

It follows from~\eqref{eq:argument},~\eqref{eq:condition3},~\eqref{eq:calculation1}, and~\eqref{eq:calculation3} that 
$$ 
\limsup_{n\to\infty}\frac{P_X(n)}{n}\leq d+1 
$$ 
and therefore is equal to $d+1$ by Theorem~\ref{theorem:main}.  This establishes the theorem. 
\end{proof} 

We end with several constructions showing various senses in which our results cannot be improved.  We first review 
some standard facts about Sturmain shifts.  A {\em Sturmian shift} $(Y, \sigma)$ is a minimal subshift of $\{0,1\}^{\Z}$ whose complexity function satisfies $P_Y(n)=n+1$ for all $n\in\N$.   
Any Sturmian shift is uniquely ergodic, and for any $\alpha\in (0,1)\setminus\mathbb{Q}$, there exists a Sturmian 
shift $(Y_\alpha, \sigma)$ whose unique invariant probability measure $\mu$ 
satisfies $\mu([0]) = \alpha$.  In particular, there are uncountably many distinct Sturmian shifts.  

We first show that the technical condition (that there exists a generic measure $\mu$ and a generic point $x_{\mu}$ such that the orbit closure of $x_{\mu}$ is not uniquely ergodic) cannot be dropped from the second statement in Theorem~\ref{theorem:main}: 
\begin{proposition} 
For $d\geq 1$, there exists a subshift $(X,\sigma)$ which has precisely $d$ ergodic measures, zero nonergodic generic measures, and whose complexity function satisfies $P_X(n)=dn+d$ for all $n\in\N$.  This subshift has the property that every $x\in X$ is generic for some ergodic measure and the orbit closure of any point is uniquely ergodic.  
\end{proposition} 
\begin{proof} 
Fix $d\in\N$ and fix a Sturmian shift $(Y,\sigma)$ on the alphabet $\{0,1\}$.  Let $\mathcal{A}:=\{0_1,1_1,0_2,1_2,\dots,0_d,1_d\}$ and for $1\leq i\leq d$ let $Y_i\subset\A^{\Z}$ be the image of $(Y,\sigma)$ under the $1$-block code that sends $0\mapsto0_i$ and $1\mapsto1_i$.  Let 
$$ 
X:=\bigcup_{i=1}^dY_i\subset\A^{\Z} 
$$ 
and observe that $X$ is closed and $\sigma$-invariant.  Moreover, we have $P_X(n)=dn+d$ for all $n\in\N$.  Each subshift $Y_i\subset X$ supports a unique ergodic measure and so there are at least $d$ ergodic measures for $(X,\sigma)$.  Conversely, for each $x\in X$ there exists $1\leq i\leq d$ such that $x\in Y_i$.  Since $Y_i$ is uniquely ergodic, $x$ is generic for the (unique) ergodic measure supported on $Y_i$.  Thus there can be no other measures that have a generic point. 
\end{proof} 

Finally we show that the assumption of linear growth in Theorem~\ref{th:upper} is optimal, in the sense that there is no analog of Theorem~\ref{th:upper} with an assumption of a superlinear growth rate and conclusion that the set of ergodic measures is finite for all subshifts whose complexity function grows at most at that rate.  

\begin{proposition} 
Let $(p_n)_{n=1}^\infty$ be a sequence of real numbers such that  
$$ 
\liminf_{n\to\infty}\frac{p_n}{n}=\infty. 
$$ 
Then there exists a subshift $(X,\sigma)$ which has infinitely many nonatomic ergodic measures and is such that for all but finitely many $n$, we have $P_X(n)\leq p_n$. 
\end{proposition} 
\begin{proof} 
For each $n\in\N$, there exists a set $\mathcal{F}_n\subset\{0,1\}^n$ such that $|\mathcal{F}_n|=n+1$ and for uncountably many $\alpha\in(0,1)$, we have $\mathcal{L}_n(Y_{\alpha})=\mathcal{F}_n$.  For $N\leq n$, let $\mathcal{X}_N(\mathcal{F}_n)$ be the set of words of length $N$ that arise as a subword of a word in $\mathcal{F}_n$.  Clearly if $\mathcal{L}_n(Y_{\alpha})=\mathcal{F}_n$ then $\mathcal{L}_N(Y_{\alpha})=\mathcal{X}_N(\mathcal{F}_n)$.  Let $\mathcal{G}_1\subset\{0,1\}$ be such that for infinitely many $n\in\N$ we have $\mathcal{G}_1=\mathcal{X}_1(\mathcal{F}_n)$.  Inductively, we assume that we have defined $\mathcal{G}_i\in\{0,1\}^i$ for all $1\leq i<j$ such that 
	\begin{enumerate} 
	\item For all $1\leq j_1<j_2<j$ we have $\mathcal{G}_{j_1}=\mathcal{X}_{j_1}(\mathcal{G}_{j_2})$; 
	\item There are infinitely many $n$ for which $\mathcal{G}_{j-1}=\mathcal{X}_{j-1}(\mathcal{F}_n)$.   
	\end{enumerate} 
We then choose $\mathcal{G}_j\in\{0,1\}^j$ such that among those $n$ for which $\mathcal{G}_{j-1}=\mathcal{X}_{j-1}(\mathcal{F}_n)$, there are infinitely many $n$ for which $\mathcal{G}_j=\mathcal{X}_j(\mathcal{F}_n)$.  
In this way, we obtain an infinite sequence $\mathcal{G}_1,\mathcal{G}_2,\dots$ such that if $1\leq j_1<j_2$, 
then $\mathcal{G}_{j_1}=\mathcal{X}_{j_1}(\mathcal{G}_{j_2})$ and there are uncountably many $\alpha\in(0,1)$ for which $\mathcal{L}_{j_2}(Y_{\alpha})=\mathcal{G}_{j_2}$. 

For each $n\in\N$, set 
$$ 
A_n:=\{\alpha\in(0,1)\colon\mathcal{L}_n(Y_{\alpha})=\mathcal{G}_n\}. 
$$ 
Then by construction, $A_n$ is uncountable for all $n\in\N$,  
$$ 
A_1\supseteq A_2\supseteq A_3\supseteq\cdots 
$$ 
and for infinitely many $n\in\N$ we have $A_n\neq A_{n+1}$.  (If not, there exist distinct $\alpha_1,\alpha_2\in\cap A_n$, and so $\mathcal{L}_n(Y_{\alpha_1})=\mathcal{L}_n(Y_{\alpha_2})$,  contradicting the fact that the frequency with which the letter $0$ occurs as a subword of any word in $\mathcal{L}_n(Y_{\alpha_i})$ tends to $\alpha_i$ for $i=1,2$.)

We now construct the subshift.  Find $N_0\in\N$ such that for all $n\geq N_0$ we have $p_n>n+1$.  Fix $\alpha_1\in A_1$ and set $X_1:=Y_{\alpha_1}$.  Then $P_{X_1}(n)<p_n$ for all $n\geq N_0$.  
Find $N_1\in\N$ such that for all $n\geq N_1$ we have $p_n>2n+2$.  
Choose the smallest $M_1\geq N_1$ for which $A_{M_1+1}\neq A_{M_1}$ and let 
$\alpha_2\in A_{M_1}\setminus A_{M_1+1}$.  Then $\mathcal{L}_{M_1}(Y_{\alpha_2})=\mathcal{L}_{M_1}(Y_{\alpha_1})$, but $\mathcal{L}_{M_1+1}(Y_{\alpha_2})\neq\mathcal{L}_{M_1+1}(Y_{\alpha_1})$.  Set $X_2:=Y_{\alpha_1}\cup Y_{\alpha_2}$.  Then $P_{X_2}(n)=n+1$ for all $n\leq M_1$, but $n+1<P_{X_2}(n)\leq 2n+2$ for all $n>M_1$.  In particular, $P_{X_2}(n)\leq p_n$ for all $n\geq N_0$.  
Now recursively, suppose we have chosen $\alpha_1,\dots,\alpha_i$ in such a way that 
$X_i:=Y_{\alpha_1}\cup\cdots\cup Y_{\alpha_i}$ satisfies $P_{X_i}(n)\leq\min\{p_n,(i-1)n+(i-1)\}$ for all $n\geq N_0$.  Find $N_i\in\N$ such that for all $n\geq N_i$, we have $p_n>in+i$.  
Find $M_i\geq N_i$ such that $A_{M_i+1}\neq A_{M_i}$ and let $\alpha_{i+1}\in A_{M_i}\setminus A_{M_i+1}$ be distinct from $\alpha_1,\dots,\alpha_i$.  
Then $\mathcal{L}_{M_i}(Y_{\alpha_{i+1}})\subseteq\mathcal{L}_{M_i}(X_i)$ but $\mathcal{L}_{M_i+1}(Y_{\alpha_{i+1}})\not\subseteq\mathcal{L}_{M_i+1}(X_i)$.  Set $X_{i+1}:=X_i\cup Y_{\alpha_{i+1}}$.  Then $P_{X_{i+1}}(n)=P_{X_i}(n)$ for all $n\leq M_i$ and $P_{X_i}(n)<P_{X_{i+1}}(n)\leq in+i$ for all $n\geq M_i$.  Thus we obtain a sequence of subshifts 
$$ 
X_1\subset X_2\subset X_3\subset\cdots 
$$ 
such that for all $i\in\N$ and all $n\geq N_0$, we have $P_{X_i}(n)<p_n$.  Setting 
$$ 
X:=\overline{\bigcup_{i=1}^{\infty}X_i}, 
$$ 
we have that $\mathcal{L}_n(X)=\bigcup_{i=1}^{\infty}\mathcal{L}_n(X_i)$ for all $n\in\N$.  Therefore, for all $n\geq N_0$, the complexity satisfies $P_X(n)<p_n$.  On the other hand, for all $i\in\N$ we have $Y_{\alpha_i}\subset X$ and there is an ergodic probability supported on $Y_{\alpha_i}$.  Since $Y_{\alpha_i}\neq Y_{\alpha_j}$ for all $i\neq j$ by construction, $X$ has infinitely many ergodic measures. 
\end{proof}

\end{document}